\documentclass[10pt, reqno]{amsart}
\usepackage{amsmath, amsthm, amscd, amsfonts, amssymb, graphicx, color}
\usepackage[bookmarksnumbered, colorlinks, plainpages]{hyperref}
\hypersetup{colorlinks=true,linkcolor=red, anchorcolor=green, citecolor=cyan, urlcolor=red, filecolor=magenta, pdftoolbar=true}
\usepackage{mathrsfs}

\newtheorem{theorem}{Theorem}[section]
\newtheorem{lemma}[theorem]{Lemma}

\newtheorem{corollary}[theorem]{Corollary}
\theoremstyle{definition}

\theoremstyle{remark}
\newtheorem{remark}[theorem]{Remark}
\numberwithin{equation}{section}

\begin{document}
\setcounter{page}{1}

\title[Class number zeta function of imaginary quadratic fields]{Class number zeta function of imaginary quadratic fields}

\author[Nikolaev]
{Igor V. Nikolaev$^1$}

\address{$^{1}$ Department of Mathematics and Computer Science, St.~John's University, 8000 Utopia Parkway,  
New York,  NY 11439, United States.}
\email{\textcolor[rgb]{0.00,0.00,0.84}{igor.v.nikolaev@gmail.com}}

\dedicatory{All data are available as part of the manuscript}

\subjclass[2020]{Primary 11M55; Secondary 46L85.}

\keywords{zeta functions, Drinfeld modules, noncommutative tori.}


\begin{abstract}
We introduce a zeta function counting  imaginary quadratic number fields by their class numbers. 
It is proved that such a function is rational  depending only on the eight roots of unity of degrees $1$ and $2$. 
As a corollary, one gets a lower bound $2p$ for the number of imaginary quadratic fields of the prime class number $p$. 
Our method  is  based on the study of  periodic points of a  dynamical system
 arising in the representation theory of the Drinfeld modules by the bounded 
linear operators  on a Hilbert space.   
\end{abstract}

\maketitle

\section{Introduction}
Classification of the imaginary quadratic number fields $\mathcal{Q}$  by their
class numbers $h$ dates back to [Gauss 1801] \cite[Article 304]{G}; we refer the reader to  the survey [Stark 2007] \cite{Sta1}
for an update.  Let $\# h$ be the cardinality 
of a subset of $\mathcal{Q}$ consisting of fields of the class number $h$.
It is known that $\# h$ is a finite number defined for any $h\ge 1$  [Heilbronn 1934] \cite[Theorem I]{Hei1}. 
The aim of our note is a zeta function given by the Lambert series or, equivalently, by the Euler product:
\begin{equation}\label{eq1.1}
\zeta_{\mathcal{Q}}(s):=\exp\left(\sum_{h=1}^{\infty} \frac{\# h}{h} \frac{s^h}{1-s^h}\right)=
\prod_{h=1}^{\infty} \frac{1}{(1-s^h)^{\frac{\# h}{h}}} , \quad s\in\mathbf{C}. 
\end{equation}

\medskip
Our main result is a rationality of the function $\zeta_{\mathcal{Q}}(s)$ given by the following 
formula. 
\begin{theorem}\label{thm1.1}
\begin{equation}\label{eq1.2}
\zeta_{\mathcal{Q}}(s)=
\frac{(1+s^2)(1-s^6)}{(1-s)^8}, \quad s\in\mathbf{C}.
\end{equation}
\end{theorem}
\begin{remark}\label{rmk1.2}
The set $\mathcal{Q}$ precludes some  fields, e.g. $\mathbf{Q}(\sqrt{-1})$ \cite[Remark 1.3]{Nik2}; hence the actual value
of $\# h$ is higher than predicted by formulas (\ref{eq1.1}) and (\ref{eq1.2}). 
Moreover, Dold's Theorem \ref{thm2.3} implies  the ratio $\frac{\# h}{h}$ in
(\ref{eq1.1}) to be integer. Further,
the single pole $s=1$ of function (\ref{eq1.2}) has order $8$.  
Likewise, all  zeros of  (\ref{eq1.2}) are roots of unity of degrees $1$ and $2$;
this fact   can be viewed as an analog of the Riemann hypothesis for the zeta function $\zeta_{\mathcal{Q}}(s)$. 
\end{remark}

For the sake of clarity, let us  review main ideas; we refer the reader to Section 2 for the notation 
and details.  Let  $F: Drin_A^{r}(\mathfrak{k})\mapsto \mathscr{A}_{RM}^{2r}$ be a functor
from the category of Drinfeld  modules of rank $r\ge 1$   to a category 
of the noncommutative tori of dimension $2r$ \cite[Theorem 3.3]{Nik1}. 
Denote by  $\Lambda_{\rho}[a]$  the torsion submodule of  the $A$-module  $\overline{\mathfrak{k}_{\rho}}$.
It is known that  $F(\Lambda_{\rho}[a])=\{e^{2\pi i\alpha_j+\log\log\varepsilon} ~|~1\le j\le r\}$,
where $\varepsilon$ is a unit of the number field $\mathbf{Q}(\alpha_j)$ of degree $2r$ over $\mathbf{Q}$.  
Let  $k$ is the maximal subfield of the number field $K=\mathbf{Q}(F(\Lambda_{\rho}[a]))$ fixed by the action of 
all elements of the group $G\subseteq GL_r\left(A/aA\right)$ and let $k\subset(\mathbf{C} - \mathbf{R})\cup\mathbf{Q}$. 
In this case  the number field 
$K\cong k\left(e^{2\pi i\alpha_j +\log\log\varepsilon}\right)$
is a Galois extension of the field  $k\cong\mathbf{Q}(i\alpha_j)$,  such that  $Gal~(K |k)\cong G$ \cite[Corollary 3.4]{Nik1}. 
In particular,  $k\cong K$,  if and only if,   $i\alpha_j=e^{2\pi i\alpha_j +\log\log\varepsilon}$. 
In other words, one gets  $k\cong K$,  if and only if,  $z_j=2\pi i\alpha_j$ is a fixed point of the map 
$f(z_j)=\lambda e^{z_j}$, where $\lambda=2\pi\log\varepsilon$. 
On the other hand, the class group $Cl~(k)\cong Gal~(K|k)\subseteq GL_r\left(A/aA\right)$  is trivial, if and only if,  
 $h:=|Cl~(k)|=1$.  Thus fixed points of the map  $f(z_j)=\lambda e^{z_j}$ are
  counting  the number $\# h$ of  fields $k$ with $h=1$.
In general, if $h\ge 1$,  then  $\# h$  is equal to the number of the least $h$-periodic points 
of the map $f(z_j)$ (Lemma \ref{lm3.1}).

Let $k$ be an imaginary quadratic field. Since the rank of Drinfeld module $r=1$,
one deals with a single function $f(z)=\lambda e^z$, where $\lambda=2\pi\log\varepsilon$
and $\varepsilon\in U$. The set $U:=\{\pm 1, \pm i, \frac{\pm 1\pm i\sqrt{3}}{2}\}$
consists of the eight roots of unity of degrees $1$ and $2$. 
Consider  a  function $f(z,\varepsilon)$ defined on the Riemann sphere $\mathbf{C}\cup\infty$. Such a function
 admits  a uniformization $\widetilde{f}$ on the double cover of the sphere by the disjoint union of four 
copies of complex tori $\mathbf{C}/(\mathbf{Z}+\mathbf{Z}\tau)$ (Figure 1). 
It is proved that the number of  the $h$-periodic  points of   $\widetilde{f}$ is equal to 
such of the Gr\"ossencharacter $\psi(\mathscr{P})$ on $\mathbf{C}/(\mathbf{Z}+\mathbf{Z}\tau)$
as $\mathscr{P}\to 1$ (Lemma \ref{lm3.3}). 
In particular,  the known rationality of the local zeta function
implies  the function $\zeta_{\mathcal{Q}}(s)$ is rational. 
Theorem \ref{thm1.1} follows. 

In view of Remark \ref{rmk1.2}, one can use formulas (\ref{eq1.1}) and (\ref{eq1.2}) to get  
the lower bounds of  number $\# h$.  Namely, let the symbol $\lessapprox$ denote an approximately less relation 
for the integers.  The following is true. 
\begin{corollary}\label{cor1.3}
If $\# h$ is the number of imaginary quadratic fields of class number $h$, then:

\medskip
(i) $d h\lessapprox \# h$ for some $d\in\{1,2, 3,\dots\}$ and hence
$h\lessapprox \# h$;

\smallskip
(ii) if  $h=p$ is a prime number, then $2p\lessapprox \# p$ and $d\ge 3$
 for $h\ne p$.
\end{corollary}

\begin{remark}
The reader can find in Figure 3 an illustration of Corollary \ref{cor1.3} by a data for $h\le 100$ due to
[Watkins 2004] \cite[Table 4]{Wat1}. The prime values of $h$ and the corresponding value of $\# h$ 
are marked in a box. 
\end{remark}

\medskip
The paper is organized as follows.  A brief review of the preliminary facts is 
given in Section 2. Theorem \ref{thm1.1} and   Corollary \ref{cor1.3} 
are proved in Section 3.

\section{Preliminaries}
We briefly review the  noncommutative tori, non-abelian class field theory and dynamical zeta function. 
We refer the reader to  [Rieffel 1990] \cite{Rie1},  [Rosen 2002] \cite[Chapters 12 \& 13]{R},  
 [Smale 1967] \cite[I. 4]{Sma1} and \cite{Nik1}  for a detailed exposition.

\subsection{Noncommutative tori}
The $C^*$-algebra is an algebra  $\mathscr{A}$ over $\mathbf{C}$ with a norm 
$a\mapsto ||a||$ and an involution $\{a\mapsto a^* ~|~ a\in \mathscr{A}\}$  such that $\mathscr{A}$ is
complete with  respect to the norm, and such that $||ab||\le ||a||~||b||$ and $||a^*a||=||a||^2$ for every  $a,b\in \mathscr{A}$.  
Each commutative $C^*$-algebra is  isomorphic
to the algebra $C_0(X)$ of continuous complex-valued
functions on some locally compact Hausdorff space $X$. 
Any other  algebra $\mathscr{A}$ can be thought of as  a noncommutative  
topological space.

By $M_{\infty}(\mathscr{A})$ 
one understands the algebraic direct limit of the $C^*$-algebras 
$M_n(\mathscr{A})$ under the embeddings $a\mapsto ~\mathbf{diag} (a,0)$. 
The direct limit $M_{\infty}(\mathscr{A})$  can be thought of as the $C^*$-algebra 
of infinite-dimensional matrices whose entries are all zero except for a finite number of the
non-zero entries taken from the $C^*$-algebra $\mathscr{A}$.
Two projections $p,q\in M_{\infty}(\mathscr{A})$ are equivalent, if there exists 
an element $v\in M_{\infty}(\mathscr{A})$,  such that $p=v^*v$ and $q=vv^*$. 
The equivalence class of projection $p$ is denoted by $[p]$.   
We write $V(\mathscr{A})$ to denote all equivalence classes of 
projections in the $C^*$-algebra $M_{\infty}(\mathscr{A})$, i.e.
$V(\mathscr{A}):=\{[p] ~:~ p=p^*=p^2\in M_{\infty}(\mathscr{A})\}$. 
The set $V(\mathscr{A})$ has the natural structure of an abelian 
semi-group with the addition operation defined by the formula 
$[p]+[q]:=\mathbf{diag}(p,q)=[p'\oplus q']$, where $p'\sim p, ~q'\sim q$ 
and $p'\perp q'$.  The identity of the semi-group $V(\mathscr{A})$ 
is given by $[0]$, where $0$ is the zero projection. 
By the $K_0$-group $K_0(\mathscr{A})$ of the unital $C^*$-algebra $\mathscr{A}$
one understands the Grothendieck group of the abelian semi-group
$V(\mathscr{A})$, i.e. a completion of $V(\mathscr{A})$ by the formal elements
$[p]-[q]$.  The image of $V(\mathscr{A})$ in  $K_0(\mathscr{A})$ 
is a positive cone $K_0^+(\mathscr{A})$ defining  the order structure $\le$  on the  
abelian group  $K_0(\mathscr{A})$. The pair   $\left(K_0(\mathscr{A}),  K_0^+(\mathscr{A})\right)$
is known as a dimension group of the $C^*$-algebra $\mathscr{A}$  [Blackadar 1986] \cite[Chapter III]{B}.

The $m$-dimensional noncommutative torus $\mathscr{A}_{\Theta}^m$ is the
universal $C^*$-algebra  generated by unitary operators $u_1,\dots, u_m$
satisfying the commutation relations 
\begin{equation}\label{eq2.1}
u_ju_i=e^{2\pi i \theta_{ij}} u_iu_j, \quad 1\le i,j\le m
\end{equation}
for a skew-symmetric matrix  $\Theta=(\theta_{ij})\in M_m(\mathbf{R})$
[Rieffel 1990] \cite{Rie1}. 
 It is known that 
  $K_0(\mathscr{A}_{\Theta}^m)\cong K_1(\mathscr{A}_{\Theta}^m)\cong \mathbf{Z}^{2^{m-1}}$.
The canonical trace $\tau$ on the $C^*$-algebra
$\mathscr{A}_{\Theta}^m$ defines a homomorphism from 
$K_0(\mathscr{A}_{\Theta}^m)$ to the real line $\mathbf{R}$;
under the homomorphism, the image of $K_0(\mathscr{A}_{\Theta}^m)$
is a $\mathbf{Z}$-module, whose generators $\tau=(\tau_i)$ are polynomials 
in $\theta_{ij}$.  The noncommutative torus  $\mathscr{A}_{\Theta}^m$ is said 
to have real multiplication if all $\theta_{ij}$ are algebraic numbers;  in  this case we use notation
$\mathscr{A}_{RM}^m$. The  positive cone
 is given by the formula  $K_0^+(\mathscr{A}_{RM}^{m})\cong \mathbf{Z}+\alpha_1\mathbf{Z}+\dots+
\alpha_{m}\mathbf{Z}\subset \mathbf{R}$, where $\alpha_j\in\mathbf{R}$ are algebraic integers of degree $m$ over $\mathbf{Q}$.

\subsection{Non-abelian class field theory}
Let  $\mathfrak{k}:=\mathbf{F}_q(T)$ ($A:=\mathbf{F}_q[T]$, resp.) be the field of rational functions (the ring of polynomial functions, resp.)
in one variable $T$ over a finite field $\mathbf{F}_q$, where $q=p^n$
and let  $\tau_p(x)=x^p$. 
Recall that  the  Drinfeld module $Drin_A^{r}(\mathfrak{k})$   of rank $r\ge 1$
is a homomorphism
\begin{equation}\label{eq2.2}
\rho:  ~A\buildrel r\over\longrightarrow \mathfrak{k}\langle\tau_p\rangle
\end{equation}
given by a polynomial $\rho_a=a+c_1\tau_p+c_2\tau_p^2+\dots+c_r\tau_p^r$ with $c_i\in \mathfrak{k}$ and $c_r\ne 0$, 
such that for all $a\in A$ the constant term of $\rho_a$ is $a$ and 
$\rho_a\not\in \mathfrak{k}$ for at least one $a\in A$ [Rosen 2002] \cite[p. 200]{R}.
For each non-zero $a\in A$ the function 
field $\mathfrak{k}\left(\Lambda_{\rho}[a]\right)$  is a Galois extension of $\mathfrak{k}$,
such that its  Galois group is isomorphic to a subgroup $G$ of the matrix group $GL_r\left(A/aA\right)$,
where   $\Lambda_{\rho}[a]=\{\lambda\in\overline{ \mathfrak{k}} ~|~\rho_a(\lambda)=0\}$
is a torsion submodule of the non-trivial  Drinfeld module  $Drin_A^{r}(\mathfrak{k})$  [Rosen 2002] \cite[Proposition 12.5]{R}.
Clearly, the abelian extensions correspond to the case $r=1$.

Let $G$ be a  left cancellative  semigroup generated by $\tau_p$ and all  $a_i\in \mathfrak{k}$ subject to the commutation relations 
$\tau_p a_i=a_i^p\tau_p$. In other words, we omit the additive structure and consider a multiplicative semigroup of the ring $\mathfrak{k}\langle\tau_p\rangle$.
  Let $C^*(G)$ be the semigroup $C^*$-algebra [Li 2017] \cite{Li1}.  
For a Drinfeld module  $Drin_A^{r}(\mathfrak{k})$  defined  by  (\ref{eq2.6}) we consider a homomorphism of the semigroup $C^*$-algebras:  
\begin{equation}\label{eq2.3}
C^*(A)\buildrel r\over\longrightarrow C^*(\mathfrak{k}\langle\tau_p\rangle). 
\end{equation}
It is proved that (\ref{eq2.3}) defines a map  $F: Drin_A^{r}(\mathfrak{k})\mapsto \mathscr{A}_{RM}^{2r}$ \cite[Definition 3.1]{Nik1}. 
\begin{theorem}\label{thm2.1}
{\bf (\cite[Theorem 3.3]{Nik1})}
The following is true:

\medskip
(i) the map $F: Drin_A^{r}(\mathfrak{k})\mapsto \mathscr{A}_{RM}^{2r}$ is a functor 
from the category of Drinfeld  modules $\mathfrak{D}$ to a category 
of the noncommutative tori $\mathfrak{A}$,   which maps any pair of isogenous  (isomorphic, resp.) 
modules  $Drin_A^{r}(\mathfrak{k}), ~\widetilde{Drin}_A^{r}(\mathfrak{k})\in \mathfrak{D}$
to a pair of the homomorphic (isomorphic, resp.)  tori  $\mathscr{A}_{RM}^{2r}, \widetilde{\mathscr{A}}_{RM}^{2r}
\in \mathfrak{A}$;  

\smallskip
(ii) $F(\Lambda_{\rho}[a])=\{e^{2\pi i\alpha_i+\log\log\varepsilon} ~|~1\le i\le r\}$,
where $\mathscr{A}_{RM}^{2r}=F(Drin_A^{r}(\mathfrak{k}))$, 
$\alpha_i$ are generators of the Grothendieck semi-group $K_0^+(\mathscr{A}_{RM}^{2r})$,  $\log\varepsilon$ is a scaling factor
 and $\Lambda_{\rho}(a)$ is the  torsion submodule of the $A$-module $\overline{\mathfrak{k}_{\rho}}$;

\smallskip
(iii) the Galois group $Gal \left(k(e^{2\pi i\alpha_i+\log\log\varepsilon})  ~| ~k\right)\subseteq GL_{r}\left(A/aA\right)$,
where $k$ is a subfield of the number field $\mathbf{Q}(e^{2\pi i\alpha_i+\log\log\varepsilon})$. 
 \end{theorem}
Theorem \ref{thm2.1} implies a non-abelian class field theory as follows.
Fix a non-zero $a\in A$ and let $G:=Gal~(\mathfrak{k}(\Lambda_{\rho}[a]) ~|~ \mathfrak{k})\subseteq GL_r(A/aA)$,
where  $\Lambda_{\rho}[a]$ is the torsion submodule of the $A$-module  $\overline{\mathfrak{k}_{\rho}}$.
Consider the number field $K=\mathbf{Q}(F(\Lambda_{\rho}[a]))$. 
Denote by $k$ the maximal subfield of $K$ which is fixed by the action of 
all elements of the group $G$. 
\begin{corollary}\label{cor2.2} 
{\bf (Non-abelian class field theory)} 
The number field
\begin{equation}\label{eq2.4}
\mathbf{K}\cong
\begin{cases} k\left(e^{2\pi i\alpha_j +\log\log\varepsilon}\right), & \hbox{if} ~k\subset(\mathbf{C} - \mathbf{R})\cup\mathbf{Q},\cr
               k\left(\cos 2\pi\alpha_j \times\log\varepsilon\right), & \hbox{if} ~k\subset\mathbf{R},
\end{cases}               
\end{equation}
is a Galois extension of  $k$,
 such that  $Gal~(K |k)\cong G$.
\end{corollary}

\subsection{Dynamical zeta function}
Let $f:M\to M$ be a diffeomorphism of a smooth manifold $M$. 
We assume that the number $N_m$ of fixed points of $f^m$ is finite 
for all $m=1, 2,  3, \dots$  The Artin-Mazur zeta function of $f$ is defined
by the formal power series:
\begin{equation}\label{eq2.5}
\zeta_{f}(s)=\exp\left(\sum_{m=1}^{\infty} \frac{N_m}{m} s^m\right), \quad s\in\mathbf{C}. 
\end{equation}
 Some fixed points of $f^m$ come from the fixed points of the lower powers of $f$. 
 We shall denote by  $K_m$ the number of periodic points of the least
 period $m$,  i.e. the ``new''  fixed points of $f^m$.     
A relation between the integers  $K_m$ and $N_m$ is given by the M\"obius inversion formula:
\begin{equation}\label{eq2.6}
K_m=\sum_{l | m}\mu(l) N_{m/l},
\end{equation}
 where 
\begin{equation}\label{eq2.7}
\mu(l)=
\begin{cases} 1, & \hbox{if} ~l=1,\cr
                     (-1)^k, & \hbox{if} ~l  ~\hbox{is product of}  ~k  ~\hbox{distinct primes},\cr 
                        0, & \hbox{if} ~l  ~\hbox{is divisible by a square number} ~>1,
\end{cases} 
\end{equation}
  is the M\"obius  function  [Smale 1967] \cite[Proposition 4.2]{Sma1}.

  \bigskip 
   The following result is due to Dold.
\begin{theorem}\label{thm2.3}
{\bf (\cite[p. 765]{Sma1})}
$K_m\equiv 0\mod ~m$. 
\end{theorem}

The Lambert series of $\zeta_f(s)$ depend on $K_m$ as follows:
\begin{equation}\label{eq2.8}
\zeta_f(s)=\exp\left(\sum_{m=1}^{\infty} \frac{K_m}{m}\frac{s^m}{1-s^m}\right), \quad s\in\mathbf{C}.
\end{equation}

The Euler product formula for $\zeta_f(s)$ can be written in the following elegant form
 [Baake,  Lau \& Paskunas 2010] \cite{BaLaPa1}:
\begin{equation}\label{eq2.9}
\zeta_f(s)=\prod_{m=1}^{\infty} \frac{1}{(1-s^m)^{\frac{K_m}{m}}} , \quad s\in\mathbf{C},
\end{equation}
where the ratio $\frac{K_m}{m}$ is always an integer (Dold's Theorem \ref{thm2.3}).


\section{Proofs}
\subsection{Proof of Theorem \ref{thm1.1}}
Let us recall the main ideas outlined in Section 1. 
First, it is shown that cardinality of the set of imaginary quadratic 
fields of class number $h$ equals   the number $K_h$ of 
the least $h$-periodic points of the map $f(z)=\lambda e^z$
(Lemma \ref{lm3.1}). 
Next, we consider  a  function $f(z,\varepsilon)$ defined on the Riemann sphere $\mathbf{C}\cup\infty$. Such a function
 admits  a uniformization $\widetilde{f}$ on the double cover of the sphere by the disjoint union of four 
copies of complex tori $\mathbf{C}/(\mathbf{Z}+\mathbf{Z}\tau)$ (Figure 1). 
It is proved that the number of  the $h$-periodic  points of   $\widetilde{f}$ is equal to 
such of the Gr\"ossencharacter $\psi(\mathscr{P})$ (Frobenius endomorphisn $Fr_p$, resp.) on 
$\mathbf{C}/(\mathbf{Z}+\mathbf{Z}\tau)$ (on the elliptic curve $\mathscr{E}(\mathbf{F}_p)$, resp.)
as $\mathscr{P}\to 1$ (Lemma \ref{lm3.3}). 
Finally, it is proved that the  zeta function $\zeta_f (s)$
coincides with the local zeta function of    $\mathscr{E}(\mathbf{F}_1)$
(Lemma \ref{lm3.4}).  We pass to a detailed argument. 

\begin{lemma}\label{lm3.1}
The number $K_h$ of the least $h$-periodic points of the map 
 $f(z)=\lambda e^z$ is equal to $\# h$, i.e. 
 the cardinality of the set of imaginary quadratic fields of
 class number $h$.  
\end{lemma} 
\begin{proof}
(i) Let $\mathbf{K}$ be the maximal unramified abelian extension of the imaginary quadratic 
field $k$.  By the class field theory, one gets $deg~(\mathbf{K}|k)=h$, where $h$ is the class number
of $k$. On the other hand, explicit formulas (\ref{eq2.4}) imply $\mathbf{K}\cong k\left(e^{2\pi i\alpha_j +\log\log\varepsilon}\right)$.  
The algebraic numbers $\{e^{2\pi i\alpha_j +\log\log\varepsilon} ~|~1\le j\le r\}$ are generators of the field $\mathbf{K}$, 
which are conjugate by the action of the Galois  group  $Gal~(\mathbf{K}|k)$.  In particular, one gets $r=h$.  

\medskip
(ii) Let $z_j=2\pi i\alpha_j$ and $\lambda=2\pi\log\varepsilon$. In view of   (\ref{eq2.4}), 
one can write:
\begin{equation}
\mathbf{K}\cong k \left(\frac{\lambda}{2\pi} e^{z_1},\dots, \frac{\lambda}{2\pi} e^{z_h}\right) \cong  k \left(\frac{1}{2\pi} f(z_1),\dots, \frac{1}{2\pi}f(z_h)\right), 
\end{equation}
where $f(z)=\lambda e^z$.  Since $\mathbf{K}$ is the maximal abelian extension of the number field $k$,
the following conditions must hold:
\begin{equation}\label{eq3.2}
\left\{
\begin{array}{ccc}
f(z_1) &\ne & z_1, \\
f(z_2) &\ne & z_2,\\
\vdots &&\\
f(z_{h-1}) &\ne & z_{h-1}, \\
f(z_h)  &=& z_1.
\end{array}
\right.
\end{equation}

\medskip
(iii) Likewise, all $z_j$ and $f(z_j)$ are of the form $2\pi$ times an algebraic number in the field $\mathbf{K}$.
There are $2h$ such numbers,  but only $h$ can be linearly independent over the field $k$. 
Without loss of generality, one  gets from (\ref{eq3.2}) the following $h$ constraints between  $z_j$ and $f(z_j)$:
\begin{equation}\label{eq3.3}
\left\{
\begin{array}{ccc}
f(z_1) &= & z_2, \\
f(z_2) &= & z_3,\\
\vdots &&\\
f(z_{h-1}) &= & z_{h}, \\
f(z_h)  &=& z_1.
\end{array}
\right.
\end{equation}
    
\medskip
(iv) It follows from (\ref{eq3.3}) that $f^j(z_1)=z_{j+1}$. 
In particular, $f^h(z_1)=z_{h+1}:=z_1$, i.e. $z_1$ is
a periodic point of the map $f$ having  the least period  $h$.
Clearly, the total number $K_h$ of such points is equal to cardinality $\# h$
of the set of imaginary quadratic fields  of class number $h$. 

\bigskip
Lemma \ref{lm3.1} is proved.   
\end{proof}

\begin{corollary}\label{cor3.2}
The Artin-Mazur zeta function of the map $f(z)=\lambda e^z$ is given by the following Lambert series:
\begin{equation}\label{eq3.4}
\zeta_{f}(s)= \exp\left(\sum_{h=1}^{\infty} \frac{\# h}{h} \frac{s^h}{1-s^h}\right), \quad s\in\mathbf{C}. 
\end{equation}
\end{corollary} 
\begin{proof}
Lemma \ref{lm3.1} says that $\# h=K_h$ for all $h\ge 1$. The conclusion of Corollary \ref{cor3.2}
follows from formua (\ref{eq2.8}) with  $m=h$.
\end{proof}

\begin{figure}
\begin{picture}(400,180)(0,40)

\put(215,60){$f(z,\varepsilon)$}

\put(60,180){$\widetilde{f}$}

\put(140,180){$\widetilde{f}$}
\put(220,180){$\widetilde{f}$}
\put(300,180){$\widetilde{f}$}


\put(60,150){\oval(60,40)}
\qbezier(50,153)(60,138)(70,153)
\qbezier(55,148)(60,153)(65,148)


\put(140,150){\oval(60,40)}
\qbezier(130,153)(140,138)(150,153)
\qbezier(135,148)(140,153)(145,148)


\put(220,150){\oval(60,40)}
\qbezier(210,153)(220,138)(230,153)
\qbezier(215,148)(220,153)(225,148)


\put(300,150){\oval(60,40)}
\qbezier(290,153)(300,138)(310,153)
\qbezier(295,148)(300,153)(305,148)


\put(140,125){\vector(1,-1){30}}
\put(225,125){\vector(-1,-1){30}}

\put(90,125){\vector(2,-1){70}}
\put(280,125){\vector(-2,-1){70}}


\put(185,65){\circle{45}}
\qbezier(165,63)(185,48)(205,63)

\end{picture}
\caption{Double cover of the Riemann sphere  $\mathbf{C}\cup\infty$ by the disjoint union of four copies of complex tori
$\mathbf{C}/(\mathbf{Z}+\mathbf{Z}\tau)$}
\end{figure}
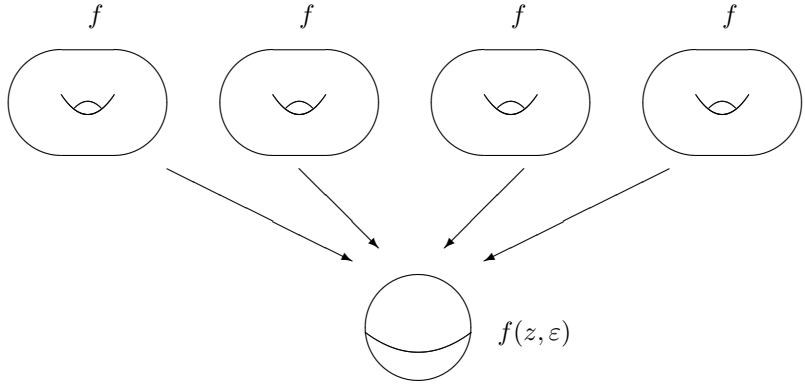

Recall that the set $U:= \{\pm 1, \pm i, \frac{\pm 1\pm i\sqrt{3}}{2}\}$  consists of all roots of unity of degrees $1$ and $2$. 
Consider an eight-valued function in the second variable $f(z,\varepsilon)=\lambda e^z$, where $\lambda=2\pi\log\varepsilon$
and $\varepsilon\in U$.  Such a function admits  a uniformization $\widetilde{f}$ on the double cover of the Riemann sphere 
$\mathbf{C}\cup\infty$ by the disjoint union of four  copies of complex tori $\{\mathbf{C}/(\mathbf{Z}+\mathbf{Z}\tau) ~|~\tau\in U\}$ as shown 
in Figure 1.  One copy of such tori corresponds to a pair of  complex conjugate values of $\tau$.
\footnote{Except for the real values $\tau=\pm 1$. This case shoud be treated aa a ``ghost'' copy corresponding to the common factor 
$1-s$ in the numerator and denominator of the rational function (\ref{eq1.2}).}
Likewise,  one can think of 
$\mathbf{C}/(\mathbf{Z}+\mathbf{Z}\tau)$  as an elliptic curve  $\mathscr{E}_{CM}$  with complex multiplication by the ring of integers of the field $\mathbf{Q}(\tau)$. 
Consider the reduction $\mathscr{E}_{CM}(\mathbf{F}_p)$  of the latter modulo the prime ideal $\mathscr{P}$ over a good prime $p$
and let $Fr_{p^m}: x\mapsto x^{p^m}$ be the Frobenius endomorphism of  $\mathscr{E}_{CM}(\mathbf{F}_p)$.  
We denote by $Fr_{1^m}$ the formal limit of $Fr_{p^m}$ as $p\to 1$.

\begin{lemma}\label{lm3.3}
The number of fixed points of the map $\widetilde{f}^m$ is 
equal to such  of the map  $Fr_{1^m}$. 
\end{lemma} 
\begin{proof}
(i) Denote by $\psi(\mathscr{P})$ the Gr\"ossencharacter associated to  elliptic curve  $\mathscr{E}_{CM}$
at  the prime ideal $\mathscr{P}$ over a good prime
 $p$. Let  $[\psi(\mathscr{P})]$ be an endomorphism of  $\mathscr{E}_{CM}$ corresponding 
 to the Frobenius endomorphism $Fr_p$ of the reduction  $\mathscr{E}_{CM}(\mathbf{F}_p)$ of   $\mathscr{E}_{CM}$
 modulo $\mathscr{P}$ [Silverman 1994] \cite[Chapter II \S 9]{S}.  Then the RHS  of diagram in Figure 2 is known to be
 commutative  [Silverman 1994] \cite[Chapter II, Proposition 10.4]{S}.

\begin{figure}
\begin{picture}(300,110)(-30,0)
\put(20,70){\vector(0,-1){35}}
\put(122,70){\vector(0,-1){35}}
\put(220,70){\vector(0,-1){35}}
\put(100,83){\vector(-1,0){60}}
\put(100,23){\vector(-1,0){60}}
\put(135,83){\vector(1,0){60}}
\put(135,23){\vector(1,0){60}}
\put(10,20){$\mathscr{A}_{RM}$}
\put(110,20){$\mathscr{E}_{CM}$}
\put(200,20){$\mathscr{E}_{CM}(\mathbf{F}_p)$}
\put(10,80){$\mathscr{A}_{RM}$}
\put(105,80){ $\mathscr{E}_{CM}$}
\put(200,80){$\mathscr{E}_{CM}(\mathbf{F}_p)$}

\put(30,50){$[L_p]$}
\put(130,50){$[\psi(\mathscr{P})]$}
\put(230,50){$Fr_p$}
\end{picture}
\caption{Gr\"ossencharacter action on noncommutative tori}
\end{figure}

\medskip
(ii) On the other hand,  the $[\psi(\mathscr{P})]$
defines an endomorphism $L_p$ of the noncommutative torus
$\mathscr{A}_{RM}$  corresponding to the $\mathscr{E}_{CM}$
\cite[Section 1.3]{N}. We shall denote by $[L_p]$
the action of $L_p$ on the group $K_0(\mathscr{A}_{RM})\cong \mathbf{Z}^2$ 
 [Blackadar 1986] \cite{B}; see the  diagram in Figure 2. 
Specifically,  $[L_p]=\left(\begin{smallmatrix} tr~[\psi(\mathscr{P})] & p\cr -1 & 0\end{smallmatrix}\right)$, where $tr$ is the trace
of $[\psi(\mathscr{P})]$ as a complex number \cite[Section 6.5.1]{N}. 
In particular, $tr ~[\psi^m(\mathscr{P})]=tr~[L_p^m]$ for every $m\ge 1$, {\it ibid.}

\medskip
(iii) By the Lefschetz fixed-point formula, one gets:
\begin{equation}\label{eq3.5}
N_m:=
|\mathscr{E}_{CM}(\mathbf{F}_{p^m})|=
\sum_{i=0}^2 (-1)^i tr ~[\psi^m(\mathscr{P})]_i=
\sum_{i=0}^2 (-1)^i tr ~[L_p^m]_i,
\end{equation}
where $tr ~[\psi^m(\mathscr{P})]_0=tr ~[L_p^m]_0=1$,
\quad
$tr ~[\psi^m(\mathscr{P})]_2=tr ~[L_p^m]_2=p^m$ 
and $tr ~[\psi^m(\mathscr{P})]_1:=tr ~[\psi^m(\mathscr{P})]= tr~[L_p^m]$. 

\medskip
(iv) Denote by $[\widetilde{f}]$ the action of $\widetilde{f}$  on the first homology  of topological torus.
Since such a  group is isomorphic to $K_0(\mathscr{A}_{RM})$, 
one  can compare    $[\widetilde{f}]$  and $[L_p]$. 
Recall that the map $f(z)=\lambda e^z$ has an inverse given by the
Lambert $W$-function \cite{Nik2}. Thus the  maps  $\widetilde{f}$ and $[\widetilde{f}]$ are 
invertible. But  $\det ~[L_p]=p$,  so that  $[L_p]$ is an invertible map 
if and only if $p=1$.  Thus one needs to  compare   $[L_1]$ with   $[\widetilde{f}]$.

\medskip
(v) Recall  that both  $[L_1]$ and   $[\widetilde{f}]$ were constructed from  the same  set of 
the arithmetic data $U=\{\pm 1, \pm i, \frac{\pm 1\pm i\sqrt{3}}{2}\}$.
Indeed, the $\mathscr{E}_{CM}\cong \mathbf{C}/(\mathbf{Z}+\tau\mathbf{Z})$ is defined by $\tau\in U$,  except for the  pair of real values. 
Thus  $[L_p]$ depends on the set $U$ for all primes including $p=1$. Likewise, $\widetilde{f}$ was constructed  from  the function $f(z,\varepsilon)$ 
by uniformization of the variable  $\varepsilon\in U$, again, except for the  pair of real values.  Thus matrix $[\widetilde{f}]$ is defined by the set $U$. 
In other words,  $[L_1]$ and   $[\widetilde{f}]$ are given by similar matrices in the group $GL_2(\mathbf{Z})$. 
In particular, $tr ~[L_1^m]=tr ~[\widetilde{f}^m]$, where  $m\ge 1$.

\medskip
(vi) The rest of proof  follows from formula (\ref{eq3.5}) when $p=1$.  Namely, one gets
$N_m= |\mathscr{E}_{CM}(\mathbf{F}_{1^m})|= \sum_{i=0}^2 (-1)^i tr ~[\widetilde{f}^m]_i$,
i.e. the number of fixed points of the map $\widetilde{f}^m$ is equal to such of the Frobenius map 
$Fr_{1^m}$. 

\bigskip
Lemma \ref{lm3.3} is proved. 
\end{proof}

\bigskip
Denote by $\{\mathscr{E}_i ~|~1\le i\le 4\}$ the connected component (an elliptic curve) of the double cover of the Riemann sphere  
$\mathbf{C}\cup\infty$ as shown in Figure 1.  Let $\zeta_{\mathscr{E}_i(\mathbf{F}_1)}(s)=\prod_{i=0}^2 \left(char~Fr_1^i\right)^{(-1)^{i+1}}$
be a (formal) local zeta function of  $\mathscr{E}_i$ at $p=1$, where $char$ is the characteristic polynomial of $i$-th Frobenius map $Fr_1^i$
with  $Fr_1^0=Fr_1^2=s-1$ and $Fr_1^1:=Fr_1$.  
\begin{lemma}\label{lm3.4}
\displaymath
\zeta_{f}(s)= \prod_{i=1}^4\zeta_{\mathscr{E}_i(\mathbf{F}_1)}(s). 
\enddisplaymath
\end{lemma} 
\begin{proof}
(i) The number $N_m$ of fixed points of the map $\widetilde{f}^m$ is given by the formula 
$N_m=\sum_{i=0}^4 N_m^i$, where $N_m^i$ is such a number for the restriction of $\widetilde{f}^m$
to the $i$-th  connected component $\mathscr{E}_i$.

\medskip
(ii)  Let us calculate the Artin-Mazur zeta function of $f$.  In view of Lemma \ref{lm3.3}, 
one gets the following formula:
\begin{equation} 
\begin{array}{lll}
\zeta_{f}(s)  = \exp\left(\sum_{m=1}^{\infty} \frac{N_m^1+\dots+N_m^4}{m} s^m\right) &=& \\
&&\\
=\exp\left(\sum_{m=1}^{\infty} \frac{\ N_m^1}{m} s^m +\dots+ \sum_{m=1}^{\infty} \frac{\ N_m^4}{m} s^m\right)&=&
\prod_{i=1}^4 \exp\left(\sum_{m=1}^{\infty} \frac{\ N_m^i}{m} s^m \right)=\\
&&\\
= \prod_{i=1}^4\zeta_{\mathscr{E}_i(\mathbf{F}_1)}(s). && \nonumber     
\end{array}
\end{equation}

\bigskip
Lemma \ref{lm3.4} is proved.
\end{proof}

\begin{corollary}\label{cor3.5}
\begin{displaymath}
\zeta_{f}(s)=
\frac{(1+s^2)(1-s^6)}{(1-s)^8}, \quad s\in\mathbf{C}.
\end{displaymath}
\end{corollary}
\begin{proof}
(i) Recall that $\zeta_{\mathscr{E}_i(\mathbf{F}_1)}(s)=\frac{char ~Fr_1^i}{(1-s)^2}$,
where $Fr_1^i$ correspond to the Gr\"ossencharacters $\psi_i(\mathscr{P})$ at the prime $p=1$,
see Figure 2. In other words,  the complex number  $[\psi(\mathscr{P})]=\tau\in \{\pm 1, \pm i, \frac{\pm 1\pm i\sqrt{3}}{2}\}$.
These values of $\tau$ are roots of  the following set of the characteristic polynomials: 
\begin{equation}\label{eq3.6}
\left\{
\begin{array}{lll}
char ~Fr_1^1&= & 1-s^2, \\
char ~Fr_1^2 &= & 1+s^2,\\
char ~Fr_1^3 &= & 1-s+s^2, \\
char ~Fr_1^4  &=& 1+s+s^2.
\end{array}
\right.
\end{equation}

 \medskip
 (ii) We  substitute  equations (\ref{eq3.6}) into the formula  given by Lemma \ref{lm3.4}
 and   we contract:
\begin{equation}\label{eq3.7}
\zeta_{f}(s)=
\frac{(1+s^2)(1-s^6)}{(1-s)^8}, \quad s\in\mathbf{C}.
\end{equation}

 \bigskip
 Corollary \ref{cor3.5} is proved.
\end{proof}

\bigskip
To finish our proof of Theorem \ref{thm1.1},  we compare (\ref{eq3.4}) and (\ref{eq3.7}) with 
the definition of the zeta function $\zeta_{\mathcal{Q}}(s)$ given by formula (\ref{eq1.1}). 

\bigskip
Theorem \ref{thm1.1} is proved.

\subsection{Proof of Corollary \ref{cor1.3}: Part I}
\begin{proof}
Corollary \ref{cor1.3} is an implication  Dold's Theorem \ref{thm2.3}
applied to Lemma \ref{lm3.1}.  

\medskip
(i) Lemma \ref{lm3.1} says that $\# h=K_h$, where $K_h$ is the number of the least $h$-periodic 
points of the map $f(z)=\lambda e^z$.  On the other hand, Theorem \ref{thm2.3} implies 
$\# h=K_h\equiv 0 \mod h$.  In other words, 
\begin{equation}\label{eq3.8}
\# h= kh, \qquad k\in\{1,2,3,\dots\}.
\end{equation}

\medskip
(ii)  We recall that the set $\mathcal{Q}$ covered by our method excludes  some 
imaginary quadratic fields, e.g. $\mathbf{Q}(\sqrt{-1})$ having class number $h=1$ \cite[Remark 1.3]{Nik2}. 
Thus, in general, equation (\ref{eq3.8}) gives us a lower bound estimation $kh\lessapprox\# h$.  

\medskip
(iii) The second statement of Part I  of Corollary \ref{cor1.3} follows from an obvious remark that  $h<kh$ for all  $k\in\{1,2,3,\dots\}$.

\bigskip
Part I of Corollary \ref{cor1.3} is proved. 
\end{proof}

\subsection{Proof of Corollary \ref{cor1.3}: Part II}
\begin{proof}
This result  follows  from
 the Euler product formula (\ref{eq2.9})  for the zeta function  $\zeta_f(s)$.
Namely,  one gets from formulas  (\ref{eq1.2}), (\ref{eq2.9})  and Lemma \ref{lm3.1}  the following identity:
\begin{equation}\label{eq3.9}
\prod_{h=1}^{\infty} \frac{1}{(1-s^h)^{\frac{\# h}{h}}} \equiv
\frac{(1+s^2)(1-s^6)}{(1-s)^8}, \quad  s\in\mathbf{C}.
\end{equation}

\medskip
(i) One can write (\ref{eq3.9}) in the  equivalent form:
\begin{equation}\label{eq3.10}
\frac{1}{(1-s)^{\# \{h=1\}}}\prod_{h=2}^{\infty} \frac{1}{(1-s^h)^{\frac{\# h}{h}}} \equiv
\frac{1}{(1-s)^8} (1+s^2)(1-s^6). 
\end{equation}
Comparing the left  and right hand side of equation (\ref{eq3.10}),
one concludes that $\# \{h=1\}=8$, i.e. the number of imaginary quadratic fields 
of class number one is equal to eight. This result agrees with \cite[item (i) of Corollary 1.2  \& Remark 1.3]{Nik2}
and as a lower bound with Watkins' Table in Figure 3.

\medskip
(ii)  After cancellation of the common factor $(1-s)^{-8}$ at the both sides of (\ref{eq3.10}) and taking the reciprocals,
one gets: 
\begin{equation}\label{eq3.11}
\begin{array}{lll}
\prod_{h=2}^{\infty} \left(1-s^h\right)^{\frac{\# h}{h}} & \equiv &
\frac{1}{(1+s^2)(1-s^6)}\approx \\
&\approx & \left(1-s^2+s^4-\dots\right)\left(1+s^6+s^{12}+\dots\right). 
\end{array}
\end{equation}

\medskip
(iii) It is easy to see, that distributing the RHS of equation (\ref{eq3.11}),  one gets monomials  of the form $\pm x^{2m_1+6m_2}$, where $m_1$ and $m_2$
are some non-negative integers.

\medskip
(iv) Let $h=p$ be a prime number at the LHS  of equation (\ref{eq3.11}).  
Then distributing  $(1-s^p)^{\frac{\# p}{p}}$, one obtains  a monomial   $\pm x^{\# p}$.
To balance the LHS and RHS of equation (\ref{eq3.11}), the latter must have the form  $\pm x^{\# p}=\pm x^{2m_1+6m_2}$,
i.e. $\# p=2(m_1+3m_2)$ for some non-negative integers $m_1$ and $m_2$.  Clearly, there are no other divisors of $\# p$ 
distinct from $2$ and $p$. We conclude therefore that  $2p\lessapprox \# p$.

\medskip
(v) The second statement in item (ii) of Corollary \ref{cor1.3}   comes from an observation that whenever  
$h\ne p$, one can expect higher degree of divisibility of $\# h$ by the distinct primes.  Therefore the ratio $\frac{\# h}{h}$
must grow up.

\bigskip
Part II of Corollary \ref{cor1.3} is proved. 
\end{proof}


\begin{figure}

\begin{tabular}{rrr rr r rr r rr r rr r rr r}
N & \# & large & N & \# & large & N & \# & large & N & \# & large & N & \# & large \\
 1 & 9 & 163 & 21 & 85 & 61483 &  \fbox{41} &  \fbox{109} & 296587 &  \fbox{61} &  \fbox{132} & 606643 & 81 & 228 & 1030723 \\
 \fbox{2} &  \fbox{18} & 427 & 22 & 139 & 85507 & 42 & 339 & 280267 & 62 & 323 & 647707 & 82 & 402 & 1446547 \\
  \fbox{3} &  \fbox{16} & 907 &  \fbox{23} &  \fbox{68} & 90787 &  \fbox{43} &  \fbox{106} & 300787 & 63 & 216 & 991027 &  \fbox{83} &  \fbox{150} & 1074907 \\
 4 & 54 & 1555 & 24 & 511 & 111763 & 44 & 691 & 319867 & 64 & 1672 & 693067 & 84 & 1715 & 1225387 \\
  \fbox{5} &  \fbox{25} & 2683 & 25 & 95 & 93307 & 45 & 154 & 308323 & 65 & 164 & 703123 & 85 & 221 & 1285747 \\
 6 & 51 & 3763 & 26 & 190 & 103027 & 46 & 268 & 462883 & 66 & 530 & 958483 & 86 & 472 & 1534723 \\
 \fbox{7} &  \fbox{31} & 5923 & 27 & 93 & 103387 &  \fbox{47} &  \fbox{107} & 375523 &  \fbox{67} &  \fbox{120} & 652723 & 87 & 222 & 1261747 \\
 8 & 131 & 6307 & 28 & 457 & 126043 & 48 & 1365 & 335203 & 68 & 976 & 819163 & 88 & 1905 & 1265587 \\
 9 & 34 & 10627 &  \fbox{29} &  \fbox{83} & 166147 & 49 & 132 & 393187 & 69 & 209 & 888427 &  \fbox{89} &  \fbox{192} & 1429387 \\
 10 & 87 & 13843 & 30 & 255 & 134467 & 50 & 345 & 389467 & 70 & 560 & 811507 & 90 & 801 & 1548523 \\
  \fbox{11} &  \fbox{41} & 15667 &  \fbox{31} &  \fbox{73} & 133387 & 51 & 159 & 546067 &  \fbox{71} &  \fbox{150} & 909547 & 91 & 214 & 1391083 \\
 12 & 206 & 17803 & 32 & 708 & 164803 & 52 & 770 & 439147 & 72 & 1930 & 947923 & 92 & 1248 & 1452067 \\
  \fbox{13} &  \fbox{37} & 20563 & 33 & 101 & 222643 &  \fbox{53} &  \fbox{114} & 425107 &  \fbox{73} &  \fbox{119} & 886867 & 93 & 262 & 1475203 \\
 14 & 95 & 30067 & 34 & 219 & 189883 & 54 & 427 & 532123 & 74 & 407 & 951043 & 94 & 509 & 1587763 \\
 15 & 68 & 34483 & 35 & 103 & 210907 & 55 & 163 & 452083 & 75 & 237 & 916507 & 95 & 241 & 1659067 \\
 16 & 322 & 31243 & 36 & 668 & 217627 & 56 & 1205 & 494323 & 76 & 1075 & 1086187 & 96 & 3283 & 1684027 \\
 \fbox{17} &  \fbox{45} & 37123 &  \fbox{37} &  \fbox{85} & 158923 & 57 & 179 & 615883 & 77 & 216 & 1242763 & \fbox{97} & \fbox{185} & 1842523 \\
 18 & 150 & 48427 & 38 & 237 & 289963 & 58 & 291 & 586987 & 78 & 561 & 1004347 & 98 & 580 & 2383747 \\
  \fbox{19} &  \fbox{47} & 38707 & 39 & 115 & 253507 &  \fbox{59} &  \fbox{128} & 474307 &  \fbox{79} &  \fbox{175} & 1333963 & 99 & 289 & 1480627 \\
 20 & 350 & 58507 & 40 & 912 & 260947 & 60 & 1302 & 662803 & 80 & 2277 & 1165483 & 100 & 1736 & 1856563 \\
\end{tabular}

\caption{ [Watkins 2004] \cite[Table 4]{Wat1}}
The table consists of five columns each containing a class number $N\le 100$, the $\#$ of negative fundamental discriminants with
class number $N$ and the absolute value of the largest such discriminant.    
 All prime values $p$  of $N$  and the corresponding value of $\# p$ are boxed out 
to illustrate the estimate $2p\lessapprox \# p$ 
(Corollary \ref{cor1.3}).
\end{figure}

.

\section*{Data availability}
  
  Data sharing not applicable to this article as no datasets were generated or analyzed during the current study.
   
\section*{Conflict of interest}
On behalf of all co-authors, the corresponding author states that there is no conflict of interest.
  

\section*{Funding declaration}
The author was partly supported by the NSF-CBMS grant 2430454.

\bibliographystyle{amsplain}


\end{document}